\documentclass[11pt,oneside]{amsart}
\usepackage{amsmath,ifthen, amsfonts, 
srcltx,
   amsopn, textcomp, txfonts, mathrsfs,}

\usepackage{hyperref}
\usepackage{graphicx}
\usepackage{xypic}
\usepackage{overpic}

\usepackage[small]{caption}
\usepackage{marginnote}
\usepackage{float}
\usepackage{tikz-cd}

\newtheorem{thm}{Theorem}[section]

\newtheorem{prop}[thm]{Proposition}

\theoremstyle{definition}

\newtheorem{rem}[thm]{Remark}

\DeclareMathOperator{\Aut}{Aut}

\newcommand{\field}[1]{\mathbb{#1}}
\newcommand{\integers}{\ensuremath{\field{Z}}}

 \label{mainB}

\newcommand\restr[2]{{% we make the whole thing an ordinary symbol
  \left.\kern-\nulldelimiterspace % automatically resize the bar with \right
  #1 % the function
  \vphantom{\big|} % pretend it's a little taller at normal size
  \right|_{#2} % this is the delimiter
  }}

\begin{document}

\author{ Daniel J.~Woodhouse}
\title{Revisiting Leighton's Theorem with the Haar Measure}

\address{Department of Mathematics \\ Technion \\ Haifa \\ Israel}
\email{daniel.woodhouse@mail.mcgill.ca}
\subjclass[]{}
\thanks{The author was supported by the Israel Science Foundation (grant 1026/15).}

\begin{abstract}
Leighton's graph covering theorem states that a pair of finite graphs with isomorphic universal covers have a common finite cover.
We provide a new proof of Leighton's theorem that allows generalizations; we prove the corresponding result for graphs with fins.
As a corollary we obtain pattern rigidity for free groups with line patterns, building on the work of Cashen-Macura and Hagen-Touikan.
To illustrate the potential for future applications, we give a quasi-isometric rigidity result for a family of cyclic doubles of free groups.
\end{abstract}

\maketitle

Leighton's graph covering theorem states that if $X_1$ and $X_2$ are finite graphs with isomorphic universal covers, then $X_1$ and $X_2$ have isomorphic finite covers.
The case of $k$-regular graphs was first proved by Angluin and Gardener~\cite{AngluinGardiner81} and was soon followed by Leighton's proof of the general case~\cite{Leighton82}.
Subsequently, Bass and Kulkarni gave another proof in the context of studying lattices in the automorphism groups of trees~\cite{BassKulkarni90}.
Walter Neumann revisited both proofs and proved a generalization for coloured graphs and partial results for what he called symmetry restricted graphs~\cite{WNeumann10}.

Given the ubiquity of group actions on trees, it is unsurprising that Leighton's theorem has seen a number of applications~\cite{BehrstockWNeumann12, Levitt15, CashenMartin17}.

Let $Y$ be a finite graph.
Let $S_i$ be a graph homeomorphic to a circle of circumference $\ell_i$, that is to say a circle subdivided into $\ell_i$ edges.
Let $\gamma_i : S_i \rightarrow Y$ be a combinatorial geodesic of length $\ell_i$, for $1 \leq i \leq n$.
A \emph{graph with fins} is the compact non-positively curved square complex $X$ obtained by taking the mapping cylinder of the map $\gamma : \bigsqcup_{i=1}^n S_i \rightarrow Y$, where $\gamma$ restricts to $\gamma_i$ on $S_i$.
Note that there is the natural retraction $X \rightarrow Y$.
We will prove the following generalization of Leighton's theorem:

\begin{thm} \label{thm:LeightonFins}
Let $X_1$ and $X_2$ be compact graphs with fins such that $\widetilde{X}_1 \cong \widetilde{X}_2$.
Then there exists isomorphic, finite index common covers of $X_1$ and $X_2$. 
\end{thm}

If our graph has no fins then we recover Leighton's original theorem.
Spiritually, we will prove Theorem~\ref{thm:LeightonFins} by walking the same path that Leighton took in the eighties.
The principal obstacle to proving Leighton style theorems is the absence of a suitable fiber product if both graphs don't cover a common base.
Beneath all the arithmetic employed in Leighton's original proof is a desire to construct something like a fiber product.
The key idea in this paper is that the numerology can be ditched and replaced with arguments involving the Haar measure.
Because the Haar measure exists for all second countable, locally compact groups, the arguments are extremely flexible and welcoming to generalization.

\subsection{Applications to rigid line patterns in free groups}

Let $X$ be a Gromov hyperbolic space.
A \emph{line pattern in $X$} is a set of equivalence classes of biinfinite quasi-geodesics up to bounded distance.
A \emph{quasi-isometry respecting line patterns} $\phi: (X, \mathcal{L}) \rightarrow (X', \mathcal{L}')$ is quasi-isometry $X \rightarrow X'$ that induces a bijection between the equivalence classes of quasi-geodesics $\mathcal{L} \rightarrow \mathcal{L}'$.

Quasi-isometries respecting line patterns were first considered by Schwartz~\cite{Schwartz97} in the context of line patterns on $\mathbb{H}^n$. 
The notion of \emph{pattern rigidity} was introduced in~\cite{MosherSageevWhyte2011}, in order to consider equivariant collections of subspaces up to quasi-isometry.
Generalizations of Schwartz's result have been given in~\cite{BiswasMj12, Mj12} and Cashen and Macura considered line patterns in free groups~\cite{CashenMacura11}, studying the problem in analogy to Whitehead's algorithm~\cite{Whitehead36}.

Let $F$ be a free group and $S = \{g_1, \ldots, g_n\} \subseteq F$ is finite set of elements of $F$ such that $g_i^p \neq h g_j^q h^{-1}$ for $1 \leq i < j \leq n$, $h \in F$, and $p,q \in \integers - \{0\}$.
That is to say that the elements in $S$ are \emph{weakly incommensurable}.
The \emph{line pattern $\mathcal{L}_S$ generated by $S$} is the equivalence classes given by the following set of quasi-geodesics:
\[
 \big\{ \; h \cdot \langle g_i \rangle \mid h \in F, \; 1 \leq i \leq n \; \big\}
\]
The line pattern $\mathcal{L}_S$ is \emph{rigid} if $F$ admits no cyclic splitting relative to $S$.
See~\cite{CashenMacura11} for full details.
If $(F, \mathcal{L}_S)$ is a free group with line pattern, and $F' \leqslant F$ is a finite index subgroup, then $F'$ inherits the line pattern $\mathcal{L}_S$ from $F$.

The following application, explained to the author by Hagen and Touikan, combines Theorem~\ref{thm:LeightonFins} with previous results to prove quasi-isometric rigidity for free groups with rigid line patterns:

\begin{thm} \label{thm:rigidLinePatterns}
 Let  $(F_1, \mathcal{L}_{S_1})$ and $(F_2, \mathcal{L}_{S_2})$ be free groups with rigid line patterns.
 Suppose that $\phi: (F_1, \mathcal{L}_{S_1}) \rightarrow (F_2, \mathcal{L}_{S_2})$ is quasi-isometry respecting the line patterns.
 Then there exists finite index subgroups $F_i' \leqslant F_i$ such that there is an isomorphism that respects the line patterns $\phi': (F_1', \mathcal{L}_{S_1}) \rightarrow (F_2', \mathcal{L}_{S_2})$.
\end{thm}

\begin{proof}
 In~\cite{CashenMacura11} Cashen and Macura construct a hyperbolic CAT(0) cube complex $X$ with line pattern $\mathcal{L}$ such that there is an action of $F_i$ on $(X, \mathcal{L})$ that is quasi-conjugate to the action of $F_i$ on $(F_i, \mathcal{L}_{S_i})$ for $i = 1,2$.
 In general, they remark that $X$ is not necessarily a tree, which would be the preferable state of affairs.
 In~\cite{HagenTouikan18} Hagen and Touikan  apply a technique called \emph{panel collapses} to show $X$ can be replaced with a tree $T$. Alternatively, this can be achieved using Mosher, Sageev, and Whyte's result about cobounded quasi-actions on bushy trees~\cite{MosherSageevWhyte2011}.
 
 For each equivalence class in $\mathcal{L}$ we attach one side of the strip $[0,1] \times \mathbb{R}$ along the unique geodesic line in the class to obtain a space $\widetilde{X}$.
 Thus the action of $F_i$ on $(T,\mathcal{L})$ becomes an action of $F_i$ on $\widetilde{X}$ for $i = 1,2$. 
 The actions are free and cocompact since they are quasi-conjugate to the original left action on $(F_i, \mathcal{L}_{S_i})$.
 The quotient $X_i = \widetilde{X} / F_i$ is a graph with fins.
 As $X_1$ and $X_2$ are graphs with fins with isomorphic universal covers we may apply Theorem~\ref{thm:LeightonFins}.
 This common finite cover corresponds to finite index subgroups $F_i' \leqslant F_i$ identified by an isomorphism that respects the line patterns. 
\end{proof}

\subsection{Applications to quasi-isometric rigidity}
The following example, suggested to the author by Emily Stark, illustrates how Theorem~\ref{thm:LeightonFins} can be used to obtain rigidity results.

\begin{thm} \label{thm:double_rigidity}
 Let $F_i$ be a finitely generated, non-abelian free group for $i = 1,2$.
 Let $w_i \in F_i$ be an element such that $\mathcal{L}_{\{w_i\}}$ is a rigid line pattern in $F_i$.
 Let ${G_i = F_i \ast_{\langle w_i \rangle} F_i}$, the amalgamated double of $F_i$ over $\langle w_i \rangle$.
 If $G_1$ is quasi-isometric to $G_2$, then $G_1$ is commensurable with $G_2$.
\end{thm}

 \begin{proof}
  As the cyclic splitting given in the statement is the JSJ-decomposition for $G_i$, Theorem 7.1 in~\cite{Papasoglu05} tells us that there is a 
 quasi-isometry respecting line patterns $$\phi: (F_1, \mathcal{L}_{\{ w_1 \}}) \rightarrow (F_2, \mathcal{L}_{\{ w_2 \}}).$$
 Hence, as in the proof of Theorem~\ref{thm:rigidLinePatterns} both $F_1$ and $F_2$ are the fundamental groups of graphs with fins $X_1$ and $X_2$ with isomorphic universal covers $\widetilde{X}_1 \cong \widetilde{X}_2$, such that the preimages of the fins correspond to the line patterns $\mathcal{L}_{S_i}$.
 Let $\widehat{X}$ be the common finite cover of $X_1$ and $X_2$ given by Theorem~\ref{thm:LeightonFins}.
 We can decompose $X_i$ as a mapping cylinder of single circle mapping into a graph $\gamma_i: S_i \rightarrow Y_i$ so  $$X_i = Y_i \sqcup S_i \times [0,1] \; \Big/ \; \big\{ (x,0) \sim \gamma_i(x) \big\}.$$
 Similarly, $\widehat{X}$ can be decomposed as a mapping cylinder of a finite set of circles mapping into a graph $\hat{\gamma}: \bigsqcup_{j=1}^n \widehat{S}_j \rightarrow \widehat{Y}$, where $\widehat{Y}$ is a finite common cover of $Y_1$ and $Y_2$. So
 $$\widehat{X} = \widehat{Y} \; \bigsqcup_{j=1}^n \widehat{S}_j \times [0,1] \; \Big/ \; \{ (\hat{x}, 0) \sim \hat{\gamma}(\hat{x}) \}.$$
 
 A graph of spaces $Z_i$ can be constructed for $G_i$ by taking two copies of $X_i$ and identifying the ends of the cylinder $S_i \times \{1\} \subseteq X_i$ and obtaining the double, so $G_i$ can be identified with $\pi_1 X_i$.
 A common finite cover $\widehat{Z}$ of $Z_1$ and $Z_2$ is constructed by similarly taking two copies of $\widehat{X}$ and identifying the ends of the cylinders $\bigsqcup_{j=1}^n \widehat{S}_j \times \{1\} \subseteq \widehat{X}$ to again obtain the double.
 Since $\widehat{X}$ covers both $X_1 $ and $X_2$, it is immediate that $\widehat{Z}$ covers both $Z_1$ and $Z_2$.
 Therefore $\pi_1 \widehat{Z}$ is a common finite index subgroup for $G_1$ and $G_2$.
 \end{proof}

{\bf Acknowledgements: } Thanks to Mark Hagen, Nicolas Touikan, Christopher Cashen, and Emily Stark for conversations, correspondence, explanations, and suggestions.

\section{Polyhedron and finite index covers} \label{sec:polyhedron}

Let $X$ be a graph with fins.
Let $Y$ be the graph underlying $X$ and assume we have fixed an orientation of each edge in $Y$.
 Let $\mathscr{H}$ be the set of \emph{vertical hyperplanes} in $X$, that is to say the hyperplanes dual to the edges in $Y$.
Let $\dot{X}$ denote the square complex obtained by subdividing along the vertical hyperplanes.

A \emph{star} is a square complex $P$ with a distinguished $0$-cube $p$ such that $P$ is the cubical neighborhood of $p$.
We can obtain stars by taking the closures of the complementary components $X - \mathscr{H}$.
The resulting stars are subcomplexes of $\dot{X}$.

A \emph{polyhedron} is a star $P$ with an isometric embedding $\phi: P \rightarrow \dot{X}$ such that $\phi(P)$ is the cubical neighborhood in $\dot{X}$ of a $0$-cube in $Y$.
A \emph{face} $(F, \varphi)$ is a finite tree $F$ that maps isomorphically to a vertical hyperplane in $X$.
We say that $(F, \varphi)$ is a \emph{face of} $(P, \phi)$ if $(F, \varphi)$ is isomorphic to the restriction of $(P, \phi)$ to a subcomplex $P \cap \phi^{-1}(\Lambda)$ where $\Lambda \in \mathscr{H}$.
Up to isomorphism, there are only finitely many polyhedron and faces.
Each face is the face of precisely two polyhedron, one of the left and one on the right.
The polyhedron on the \emph{left} [\emph{resp. right}] is the polyhedron whose image contains the origin vertex [\emph{resp. terminal vertex}] of the edge dual to the image of the face.
If $(P, \phi)$ and $(P',\phi')$ are polyhedron on the left and right of the face $(F, \varphi)$, then the polyhedron can be glued together along the subcomplexes corresponding to $F$ to obtain a new complex $P \cup P'$ that maps into $X$ via $\phi$ and $\phi'$.

An $n$-sheeted cover of $X$ can be constructed by taking $n$ copies of each polyhedron and face and then choosing a correspondence between the $n$-copies of a polyhedron on the left of a face and the $n$-copies on the right.
The polyhedron can then be glued together to obtain an $n$-sheeted cover of $\dot{X}$.
Reversing the subdivisions along the vertical hyperplanes gives the cover of $X$. See Figure~\ref{fig:polygonCovers}.

 \begin{figure}
   \begin{overpic}[width=.9\textwidth,tics=10]{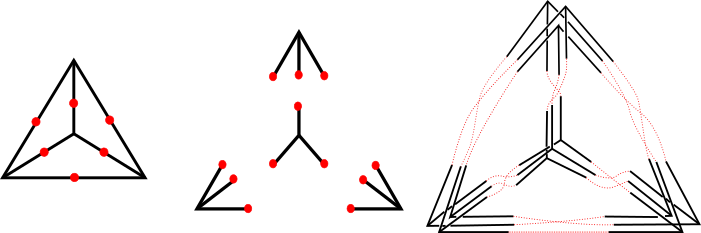} 
    \end{overpic}
    \caption{From left to right, a graph with hyperplanes highlighted, the polyhedron for the graph, and 3 copies of each polyhedron used to construct a degree three cover of the original graph.} 
    \label{fig:polygonCovers}
    \end{figure}

The idea behind the proof of Theorem~\ref{thm:LeightonFins} is to replicate this construction on two graphs with fins simultaneously.
This requires first defining an appropriate notion of being a polyhedron for two spaces simultaneously and understanding how we might glue them together appropriately.
In this setting we obtain a set of gluing equations that we must solve in order to construct our cover. 
We will utilize the Haar measure to find solutions to these equations.

\section{Proof of Main Theorem} \label{sec:MainTheorem}

Let $X_1$ and $X_2$ be graphs with fins such that $\widetilde{X}_1 \cong \widetilde{X}_2$.
Let $Y_i \subseteq X_i$ be the underlying graph and $r : X_i \rightarrow Y_i$ be the natural retraction.
Identify $\widetilde{X}_1 \cong \widetilde{X}_2 = : \widetilde{X}$ and $\widetilde{Y}_1 \cong \widetilde{Y}_2 =: \widetilde{Y}$.
Let $p_i: \widetilde{X} \rightarrow X_i$ be the covering map.
Let $\Gamma_i = \pi_1(X_i)$.
Let $\mathscr{H}_i$ be the set of \emph{vertical hyperplanes} in $X_i$.

Let $G = \Aut(\widetilde{X})$.
Then there is an embedding $\Gamma_i \leqslant G$ as the group of deck transformations.
Note that $\Gamma_i$ is a uniform free lattice in $G$.
We will assume that $G$ does not invert hyperplanes.
This can be achieved by either subdividing the vertical hyperplanes, or applying Proposition 6.3 in~\cite{Bass93} to pass to index $2$ subgroups of $G, \Gamma_1,\Gamma_2$.
Let $\underline{X} = \widetilde{X} / G$.
There are natural quotient maps $q_i : X_i \rightarrow \underline{X}$.
As $G$ does not invert hyperplanes, we can assign $G$-invariant orientations to all the hyperplanes in $\widetilde{Y}$, which gives orientations of the vertical hyperplane $\mathscr{H}_i$.
If $S$ is a subset of $\widetilde{X}$, then $G_S$ denotes the \emph{setwise stabilizer} of $S$, and $G_{(S)}$ denotes the \emph{pointwise stabilizer} of $S$.

\begin{rem} The space $\underline{X}$ encodes the information that Leighton refers to as the \emph{degree refinement}.
 Our use more closely resembles Neumann's use in~\cite{WNeumann10}, as the \emph{colouring graph}.
 By giving each vertex, edge, and square a unique colour we obtain colourings of $X_1$, $X_2$, and $\widetilde{X}$.
 The colouring of $\widetilde{X}$ is $G$-equivariant.
 In fact, everything in this section works if $G$ is the full colour preserving automorphism group of a coloured graph with fins with free uniform lattices $\Gamma_1$ and $\Gamma_2$.
\end{rem}

\subsection{Polyhedral Pairs} \label{subsec:polyhedral pair}

A \emph{polyhedral pair} is a triple $(P, \phi_1, \phi_2)$ where each pair $P, \phi_i$ is a polyhedron for $X_i$ and the following diagram commutes:

\begin{displaymath}
 \xymatrix{
 P \ar[r]^{\phi_1} \ar[d]_{\phi_2} & X_1 \ar[d]^{q_1} \\
 X_2 \ar[r]_{q_2} & \underline{X}  \\
 }
\end{displaymath}

We will denote a polyhedral pair with boldface: $\mathbf{P} = (P, \phi_1, \phi_2)$.
There are only finitely many polyhedral pairs for $X_1$ and $X_2$ up to isomorphism.
By forgetting either $\phi_1$ or $\phi_2$ we obtain a polyhedron for either $X_1$ or $X_2$. 
We say that we obtain such polyhedron by \emph{restricting $\mathbf{P}$ to either $X_1$ or $X_2$}.
See Figure~\ref{fig:polyhedral pair} for an illustrated example.

 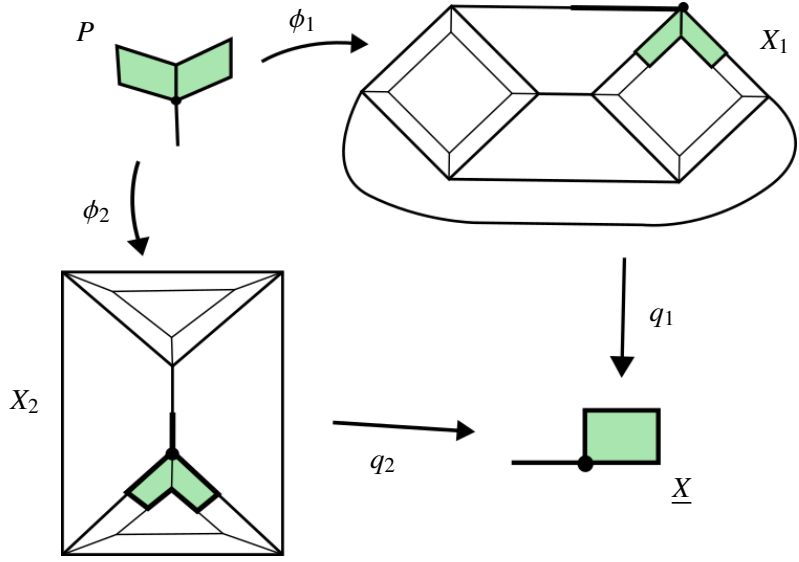
\begin{figure}
   \begin{overpic}[width=.7\textwidth,tics=10]{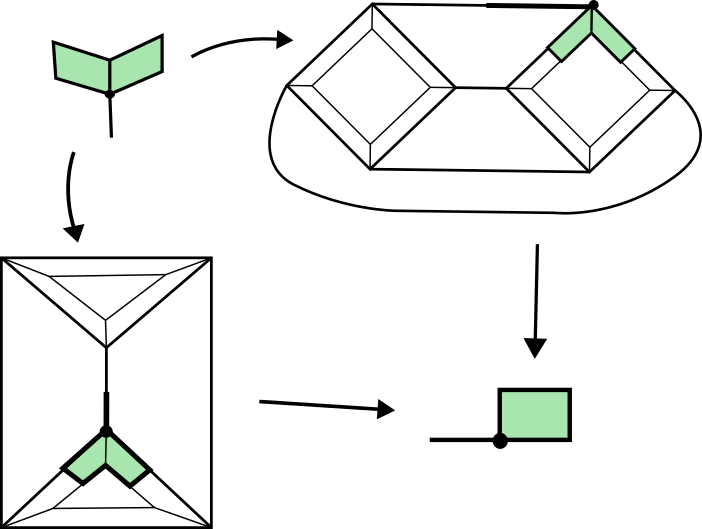} 
    \put(2,70){$P$}
    \put(95,69){$X_1$}
    \put(-7,20){$X_2$}
    \put(83,8){$\underline{X}$}
    \put(31,72){$\phi_1$}
    \put(3,46){$\phi_2$}
    \put(42,12){$q_2$}
    \put(80,32){$q_1$}
   \end{overpic}
    \caption{A polyhedral pair illustrated.} 
    \label{fig:polyhedral pair}
    \end{figure}

Let $\mathbf{P}$ be a polyhedral pair.
Let $\tilde{\phi}_i : P \rightarrow \widetilde{X}$ be some choice of lift.
Note that $\tilde{\phi}_i$ is unique up to post composition by $g_i \in \Gamma_i$.
Then a polyhedral pair $\mathbf{P}$ is \emph{admissible} if there exist a $g \in G$ such that $g \circ \tilde{\phi}_1 = \tilde{\phi}_2$.
Note that this does not depend on the initial choice of lifts.
Indeed. any alternative choice of lift is of the form $\gamma_i \circ \phi_1$ where $\gamma_i \in \Gamma_i$, so $(\gamma_2 g \gamma_1^{-1}) \circ (\gamma_1  \tilde{\phi}_1) = \gamma_2 \tilde{\phi}_2$. 
Let $\mathcal{P}$ denote the (finite) set of all admissible polyhedral pairs.
All polyhedral pairs used in this paper are admissible.

  \begin{rem}If $\widetilde{X}$ is a tree, then all polyhedral pair are admissible.
 In general, this is not true.
 \end{rem}

  We can classify admissible polyhedral pairs as follows.
  Let $(P_1, \phi_1)$ and $(P_2, \phi_2)$ be polyhedron for $X_1$ and $X_2$ such that $q_1 \circ \phi_1(P_1) = q_2 \circ \phi_2 (P_2)$.
  Let $\tilde{\phi}_i: P_i \rightarrow \widetilde{X}$ be a choice of lift for each $i = 1,2$.
  Let $v_i$ be the $0$-cube in $\widetilde{P}_i : = \tilde{\phi}_i({P}_i)$.
  Let $g\in G$ such that $g\widetilde{P}_1 = \widetilde{P}_2$.
  Such a $g$ exists since $q_1 \circ p_1(v_1) = q_2 \circ p_2(v_2)$.
  Then the admissible polyhedral pairs which restrict to the polyhedron $(P_i, \phi_i)$ are in one to one correspondence with the elements of $G_{v_1}/G_{(\widetilde{P}_1)}$.
  Indeed, for each $g' \in G_{v_1}$ we obtain an isomorphism $$\tilde{\phi}_2^{-1} \circ g  g' \circ \tilde{\phi_1}: P_1 \rightarrow P_2$$
  that allows us to identify $P_1$ and $P_2$ so that we obtain a polyhedral pair $(P, \phi_1, \phi_2)$.
  Note that all $g'$ in the same $G_{(\widetilde{P}_1)}$-coset give the same polyhedral pair since they correspond to identical identifications of $P_1$ and $P_2$.

 A \emph{face pair} $\mathbf{F} = (F, \varphi_1, \varphi_2)$ is a tuple such that each $(F, \varphi_i)$ is a face for $X_i$, and the following diagram commutes:
 
 \begin{displaymath}
 \xymatrix{
 F \ar[r]^{\varphi_1} \ar[d]_{\varphi_2} & X_1 \ar[d] \\
 X_2 \ar[r] & \underline{X}  \\
 }
\end{displaymath}
 
 Let $\tilde{\varphi_i}: F \rightarrow \widetilde{X}$ be a lift of $\varphi_i$.
 A face pair is \emph{admissible} if there exists $g \in G$ such that $g \circ \tilde{\varphi}_1 = \tilde{\varphi}_2$.
 As with polyhedral pairs, admissibility of face pairs does not depend on the choice of lifts.
 
 We say that \emph{$\mathbf{F}$ is a face of $\mathbf{P}$} if $(F, \phi_i)$ is a face of the polyhedron $(P, \phi_i)$ for $i = 1,2$.
 A polyhedral pair $\mathbf{P}$ lies on the left [\emph{resp. right}] of $\mathbf{F}$ if the polyhedron $(P, \phi_i)$ lie on the left [\emph{resp right}] of $(F, \varphi_i)$.
 Note that if $(P, \phi_1)$ lies on the left [\emph{resp. right}] of $(F, \varphi_1)$, then $(P, \phi_2)$ lies on the left [\emph{resp. right}] of $(F, \varphi_2)$, since the orientations of the edge in $Y_i$ were obtained from $G$-equivariant orientations of the edges of $\widetilde{Y}$.
 Let $\mathcal{F}$ denote the set of all admissible face pairs.
 
 Unlike faces of polyhedron, face pairs can have multiple polyhedral pairs on the left (or right).
 Let $\overleftarrow{\mathbf{F}}$ and $\overrightarrow{\mathbf{F}}$ denote the sets of polyhedral pairs on the left and right of $\mathbf{F}$.
 Given $\mathbf{P} \in \overleftarrow{\mathbf{F}}$ and $\mathbf{P}' \in \overrightarrow{\mathbf{F}}$ we can glue $\mathbf{P}$ and $\mathbf{P}'$ together along $\mathbf{F}$.
 That is to say, by identifying $F$ with the corresponding subcomplexes of $P$ and $P'$, we can glue $P$ and $P'$ together, respecting the maps $\phi_i$ and $\phi_i'$ to obtain the following commutative diagram:
 
 \begin{displaymath}
 \xymatrix{
 P\cup P' \ar[r]^{\phi_1 \cup \phi_1'} \ar[d]_{\phi_2\cup \phi_2'} & X_1 \ar[d] \\
 X_2 \ar[r] & \underline{X}  \\
 }
\end{displaymath}

Thus, performing the construction in Section~\ref{sec:polyhedron} for both $X_1$ and $X_2$ simultaneously requires taking copies of each admissible polyhedral pair such that the number on the left and right of each face is the same.

\subsection{Gluing equations} \label{subsec:gluingEquations}

We define a system of linear gluing equations.
Let $\omega: \mathcal{P} \rightarrow \mathbb{R}_{>0}$ denote a weight function on the set of admissible polyhedral pairs.
For each face $\mathbf{F}$ we have the following \emph{gluing equation}:
\begin{equation} \label{gluing_equations}
 \sum_{\mathbf{P} \in \overleftarrow{\mathbf{F}}} \omega(\mathbf{P}) = \sum_{\mathbf{P} \in \overrightarrow{\mathbf{F}}} \omega(\mathbf{P})
\end{equation}
Note that since this gives a finite set of linear equations with integer coefficients, finding a set of positive real weights satisfying all gluing equations gives a set of positive integer weights satisfying all the gluing equations.

\subsection{Solving the equation} \label{subsec:solvingEquations}

Let $\mu$ be the Haar measure for $G$.
As $G$ contains a lattice -- $\Gamma_1$ for example -- $G$ is unimodular and $\mu$ is both left and right $G$-invariant.
We recall that $\mu$ is positive on every open set and finite on every compact set.

For each polyhedral pair $\mathbf{P} = (P, \phi_1, \phi_2) \in \mathcal{P}$ recall that we let $\tilde{\phi}_i:P\rightarrow \widetilde{X}$ be a lift of $\phi_i$ and $\widetilde{P}_i = \tilde{\phi}_i(P)$ for $i = 1,2$.
Then let $$\omega(\mathbf{P}) = \mu\Big(G_{(\widetilde{P}_1)}\Big).$$
As $\mu$ is left and right invariant we deduce that $\omega(\mathbf{P})$ does not depend on the choice of lifts $\tilde{\phi}_i$.
Observe that $\omega(\mathbf{P})$ is a positive real number since it is a compact neighborhood.
Moreover, after scaling $\mu$ we can assume that $\omega(\mathbf{P}) \in \integers_{>0}$ since all stabilizers of finite sets in $\widetilde{X}$ have commensurate measures.

\begin{prop}
 The weight function $\omega$ as defined, satisfy the gluing equations~\ref{gluing_equations}.
\end{prop}

\begin{proof}
Given an admissible face $\mathbf{F}$ we need to enumerate the polyhedral pairs on the left of $\mathbf{F}$ (and right).
Let $\tilde{\varphi}_i: F \rightarrow \widetilde{X}$ be a lift of $\varphi_i$.
Let $\widetilde{F}_i = \tilde{\varphi}_i(F)$.
Let $g \in G$ be such that $g \circ \tilde{\varphi}_1 = \tilde{\varphi}_2$.

Let $(P_i, \phi_i)$ be the polyhedron on the left of $(F, \varphi_i)$.
Let $\tilde{\phi}_i : P_i \rightarrow X$ be the lift of $\phi_i$ such that $\widetilde{P}_i := \tilde{\phi}_i(P_i)$ contains $\widetilde{F}_i$.
Then each polyhedral pair to the left of $\mathbf{P}$ is obtained by taking some $g' \in G_{(\widetilde{F}_i)}$ and letting $\tilde{\phi}_2^{-1} \circ gg' \circ \tilde{\phi}_1$ identify $P_1$ and $P_2$.
All $g'$ in the same $G_{(\widetilde{P}_1)}$ coset define the same polyhedral pair.
Therefore we can compute that
\[
 \sum_{\mathbf{P} \in \overleftarrow{\mathbf{F}}} \omega(\mathbf{P}) = \sum_{ G_{(\widetilde{F}_1)}/G_{(\widetilde{P}_1)}} \mu(G_{(\widetilde{P}_1)}) \; = \; \mu\Big(G_{(\widetilde{F}_1)}\Big).
\]
A similar equality holds on the right, so the Proposition holds.
\end{proof}

\subsection{Constructing a common cover} \label{subsec:constructingCover}
Assuming that our weight function gives positive integer values, take $\omega(\mathbf{P})$ copies of each  $\mathbf{P} \in \mathcal{P}$.
Let $\mathcal{P}_{\omega}$ denote this multiset.
Let $\overleftarrow{\mathbf{F}}_{\omega}$ and $\overrightarrow{\mathbf{F}}_{\omega}$ denote the polyhedron in $\mathcal{P}_\omega$ on the left and right of $\mathbf{F}$ respectively.
 Since $\omega$ satisfies the gluing equations, for each admissible face $\mathbf{F}$ we may fix a bijection $\overleftarrow{\mathbf{F}}_\omega \rightarrow \overrightarrow{\mathbf{F}}_\omega$.
By gluing the corresponding admissible faces together we obtain a graph with fins: 
$$\widehat{X} = \bigsqcup_{\mathbf{P} \in \mathcal{P}_\omega} P \Big/ \sim$$
Then there is a covering map $\Phi: \widehat{X} \rightarrow X_i$ where $\restr{\Phi_i}{P} = \phi_i$ for each $\mathbf{P} = (P, \phi_1, \phi_2) \in \mathcal{P}_{\omega}$.
The proof of Theorem~\ref{thm:LeightonFins} is complete.

\bibliographystyle{plain}
\bibliography{ref.bib}
%\nocite{*}
\end{document}